\documentclass[final,3p,times]{elsarticle}

\usepackage{amssymb}
\usepackage{graphicx,setspace}
\usepackage{psfrag}
\usepackage{amsfonts}
\expandafter\let\csname equation*\endcsname\relax
\expandafter\let\csname endequation*\endcsname\relax
\usepackage{amsmath}
\usepackage{amssymb, amsthm}
\usepackage{mathrsfs}
\usepackage{epstopdf}
\usepackage{float}
\usepackage{lineno,hyperref}
\usepackage{color}
\usepackage{subfigure}

\usepackage{amsmath}
\usepackage{dsfont} 
\usepackage{amssymb} 
\usepackage{graphicx,color}
\usepackage{extarrows}

\usepackage{graphicx}

\usepackage{epstopdf}

\journal{Stochastics and Dynamics}

\begin{document}
\newtheorem{definition}{Definition}[section]
\newtheorem{lemma}{Lemma}[section]
\newtheorem{remark}{Remark}[section]
\newtheorem{theorem}{Theorem}[section]
\newtheorem{proposition}{Proposition}
\newtheorem{assumption}{Assumption}
\newtheorem{example}{Example}
\newtheorem{corollary}{Corollary}[section]
\def\ep{\varepsilon}
\def\Rn{\mathbb{R}^{n}}
\def\Rm{\mathbb{R}^{m}}
\def\E{\mathbb{E}}
\def\hte{\hat\theta}
\renewcommand{\theequation}{\thesection.\arabic{equation}}
\begin{frontmatter}

\title{Effective dynamics of interfaces for nonlinear SPDEs
\\
driven by multiplicative white noise}

\author{Shenglan Yuan\corref{cor1}\fnref{addr1}}\ead{shenglanyuan@gbu.edu.cn}\cortext[cor1]{Corresponding author}
\author{Dirk Bl$\rm\ddot{o}$mker\fnref{addr2}}\ead{dirk.bloemker@math.uni-augsburg.de}

\address[addr1]{\rm Department of Mathematics, School of Sciences, Great Bay University, Dongguan 523000, China}
\address[addr2]{\rm Institut f$\rm\ddot{u}$r Mathematik, Universit$\rm\ddot{a}$t Augsburg,
86135, Augsburg, Germany}

\begin{center}
\text{Dedicated to Professor Boling Guo on the occasion of his 90th birthday}
\end{center}

\begin{abstract}
In the present work, we investigate the dynamics of the infinite-dimensional stochastic partial differential equation (SPDE) with multiplicative white noise. We derive the effective equation on the approximate slow manifold in detail by utilizing a finite-dimensional stochastic differential equation (SDE) describing the motion of interfaces. In particular, we verify the equivalence between the full SPDE and the coupled system under small stochastic perturbations. Moreover, we apply our results to effective dynamics of stochastic models with multiplicative white noise, illustrated with four examples on the stochastic damped wave equation, the stochastic Allen-Cahn equation, the stochastic nonlinear Schr\"{o}dinger equation and the stochastic Swift-Hohenberg equation.
\end{abstract}


\end{frontmatter}

\section{Introduction}
The field of effective dynamics of interfaces for nonlinear SPDEs driven by multiplicative white noise lies at the intersection of stochastic analysis, nonlinear partial differential equations, and mathematical physics, focusing on understanding how multiplicative noise influences the motion, stability, and emergent behavior of interfaces (e.g., phase boundaries, propagating fronts, or domain walls) in spatially extended systems \cite{BD07}. For the nonlinear SPDEs with multiplicative noise, the noise is state-dependent, leading to  It\^{o} or Stratonovich interpretations (critical for preserving physical properties like energy balance). The nonlinear coupling between noise and solution, causing phenomena like noise-induced transitions or symmetry breaking \cite{BHP}.

The study of effective dynamics of interfaces for nonlinear SPDEs driven by multiplicative white noise has profound applications across diverse scientific domains \cite{C}. In materials science, these models capture phase separation processes in alloys, where multiplicative noise accounts for thermal fluctuations and heterogeneous material properties, enabling precise predictions of microstructural evolution and grain boundary dynamics under stochastic forcing. In fluid dynamics, such SPDEs describe stochastic fingering instabilities in multiphase flows, where multiplicative noise models spatially correlated perturbations in viscosity or pressure, offering insights into pattern formation in porous media or oil recovery processes. In biology, these equations are instrumental in analyzing cell membrane fluctuations driven by thermal or active noise, as well as population front propagation in stochastic environments, where multiplicative noise reflects spatially varying resource availability or predation risks. For climate modeling, effective interface dynamics governed by SPDEs with multiplicative noise elucidate stochastic ice-line transitions during paleoclimate shifts, capturing how random climatic forcings drive abrupt transitions between glacial and interglacial states. By reducing these multiscale systems to effective equations for interface motion, researchers gain predictive power over rare events, metastability, and critical transitions, bridging mathematics to real-world phenomena in engineering, ecology, and Earth systems.

Analyzing interface dynamics in SPDEs with multiplicative noise requires advanced probabilistic frameworks to address both regularity and noise-driven complexity \cite{H}. Stochastic calculus provides foundational tools such as It\^{o} and Stratonovich integrals, which govern the evolution of multiplicative noise in SPDEs, with Stratonovich calculus often preferred for geometric interpretations of noise in curved manifolds. Martingale problems offer a weak solution framework to characterize interface motion, linking stochastic dynamics to generator-based descriptions of Markov processes. Malliavin calculus enables sensitivity analysis by differentiating stochastic systems with respect to noise realizations, critical for quantifying how multiplicative perturbations influence interface stability or bifurcations. In regimes with rough or singular noise (e.g., spatially irregular or fractional), regularity structures and paracontrolled distributions rigorously renormalize divergent terms arising from nonlinear interactions between noise and singular solutions, ensuring well-posedness in high-dimensional or fractal geometries. These tools collectively resolve ambiguities in stochastic products and underpin convergence proofs for dimension-reduced effective models. Unifying these approaches can tackle multiscale phenomena such as stochastic front propagation in random media, crack path selection in fracture mechanics, or noise-induced pattern transitions in active matter, advancing predictive modeling of interfaces across physics, biology, and materials science.

Numerically resolving interface dynamics in SPDEs with multiplicative noise demands tailored strategies to balance accuracy, efficiency, and scalability \cite{ACLW,KW}. Stochastic discretization employs adaptive schemes, such as adaptive mesh refinement or stochastic finite elements, to resolve sharp interfaces while approximating multiplicative noise via Wiener chaos expansions or Malliavin calculus. These methods must carefully handle spatially correlated noise and singularities inherent to SPDEs, often requiring hybrid techniques to manage computational costs. Monte Carlo methods leverage statistical sampling of interface trajectories to quantify uncertainty, utilizing variance reduction techniques (e.g., importance sampling, multilevel Monte Carlo) to mitigate the curse of dimensionality and capture rare events like topological transitions or metastable jumps. Coarse-graining focuses on simulating reduced effective equations-such as SDEs for interface position or curvature-driven flows-derived via dimension reduction or homogenization. This approach benefits from structure-preserving integrators (e.g., geometric or symplectic schemes) to maintain conservation laws and stability. Key challenges include ensuring consistency between coarse-grained models and the original SPDEs, addressing noise-induced numerical stiffness, and developing data-driven methods (e.g., neural operators) to learn effective dynamics from sparse or noisy experimental data. Bridging these numerical frameworks with applications in materials science (e.g., crack propagation) or fluid dynamics (e.g., droplet breakup) requires scalable algorithms for high-dimensional systems and rigorous validation against both synthetic and experimental benchmarks \cite{KS}.

The analysis of interface dynamics in multiplicative-noise-driven SPDEs intersects with diverse disciplines, each contributing unique tools and perspectives \cite{FW}. Stochastic geometry provides frameworks to quantify topological changes (e.g., interface merging, splitting, or pinch-off) in evolving manifolds under noise, leveraging concepts like random set theory and percolation thresholds to characterize geometric transitions. Statistical mechanics bridges microscopic noise-driven fluctuations with macroscopic interface behavior, linking metastable states and nonequilibrium phase transitions to energy landscapes and entropy production rates in SPDEs. Control theory informs strategies to stabilize interfaces against multiplicative noise, employing feedback mechanisms or stochastic Lyapunov functions to suppress undesired instabilities while preserving desired morphologies. Machine learning accelerates the discovery of effective interface models \cite{LYX,LYLL} via data-driven techniques such as neural operators or physics-informed neural networks, which learn coarse-grained dynamics from high-dimensional SPDE simulations or experimental datasets. These fields synergize to address challenges like predicting fracture paths in disordered materials, controlling biofilm expansion under environmental noise, or optimizing interface shapes in stochastic optimal control. Integrating geometric intuition, thermodynamic principles, adaptive control, and data science, can unlock robust, cross-disciplinary insights into how multiplicative noise shapes interfacial phenomena across scales.

Research in this field has evolved over decades. Below is a synthesis of key notable results.
Bl\"{o}mker and Schindler \cite{BS} analyzed the stochastic motion of a droplet governed by a stochastic Cahn-Hilliard equation in the sharp-interface limit under weak noise conditions. They further derived a rigorous stochastic differential equation describing the droplet center's dynamics.
Cartwright and Gottwald \cite{CG} presented a collective coordinate approach to analyze the impact of stochastic perturbations on coherent solitary waves in Korteweg-de Vries equations.
Hairer, L\^{e} and Rosati \cite{HLR} investigated the Allen-Cahn equation with a rapidly mixing Gaussian field as the initial condition. They demonstrated that if the amplitude of this initial condition is not too large, the equation generates interfaces  characterized by the nodal sets of the Bargmann-Fock Gaussian field, which subsequently evolve according to mean curvature flow. Yuan and Bl\"{o}mker \cite{YB} reduced dynamics of stochastic amplitude equations near bifurcation points using formal expansions. Previous research established foundational tools (homogenization, invariant manifolds) and revealed critical phenomena (noise-induced transitions, modified scaling laws). However, challenges persist in handling high dimensions, strong nonlinearities, and non-Gaussian noise, motivating modern work in stochastic analysis and applied mathematics \cite{PZ}.

The central focus of current research lies in reducing SPDE dynamics onto a low-dimensional manifold representing interface motion. This approach is pivotal for analyzing complex stochastic systems by focusing on essential degrees of freedom. It arises in modeling interface dynamics for nonlinear SPDEs with multiplicative noise.

This paper is organized as follows. In Section \ref{E}, we establish an equivalence between the original SPDE and a coupled system comprising a finite-dimensional SDE for the parameter on an approximate slow manifold and an infinite-dimensional evolution equation for the orthogonal component. In Section \ref{ED}, we demonstrate the applicability of our results to the effective dynamics of stochastic models with multiplicative white noise, employing four concrete examples to elucidate their practical implications. In Section \ref{DW}, we analyze the stochastic damped wave equation with multiplicative noise, demonstrating prolonged metastability near a slow manifold.  The reduced dynamics are governed by a finite-dimensional Stratonovich SDE, validated through spectral analysis and numerical simulations. Our results establish that solutions remain close to the slow manifold for exponentially long times with high probability, even under multiplicative noise.
 In Section \ref{ACE}, we analyze metastability and noise-induced transitions in the stochastic Allen-Cahn equation with multiplicative Stratonovich noise. For spatially extended systems with Neumann boundary conditions, front solutions near a slow manifold exhibit exponentially long residence times before transitioning between metastable states. Numerical simulations validate the reduced dynamics, demonstrating rare transitions driven by multiplicative noise.
  In Section \ref{SE}, we analyze the dynamics of solitons in the stochastic nonlinear Schr\"{o}dinger equation  with multiplicative Stratonovich noise. For weak noise, we project the infinite-dimensional system onto a slow manifold corresponding to the soliton's translational mode. The reduced dynamics reveal that the soliton position undergoes a Brownian motion with a noise-induced drift. Numerical simulations confirm the diffusive motion of the soliton while maintaining coherence over long timescales, demonstrating robustness against stochastic perturbations. In Section \ref{SHE}, we investigate metastability in the stochastic Swift-Hohenberg equation  with multiplicative Stratonovich noise. For weak noise near the pattern-forming bifurcation, we project the dynamics onto a slow manifold of stripe patterns. The reduced amplitude dynamics follow a Stuart-Landau equation with noise, leading to a state-dependent stochastic differential equation. Numerical simulations confirm rare escapes from the equilibrium amplitude, consistent with theoretical predictions. In Section \ref{CF}, we summarize our findings and present our conclusions, as well as a number
of directions for future study.

\section{Equivalence of full SPDE and coupled system}\label{E}
We study SPDEs of the following type:
\begin{equation}\label{SPDE}
du(t)=\mathcal{L}(u(t))dt+G(u(t))dW(t),\quad\quad u(0)=u_0\in\mathcal{H},
\end{equation}
where $\mathcal{L}$ denotes a nonlinear differential operator. The diffusion operator $G: \mathcal{H}\rightarrow L_{HS}(X,\mathcal{H})$ maps into the space of Hilbert-Schmidt operators. The noise $W$ is an $X$-valued $\mathcal{Q}$-Wiener process dependent on both space and time variables. The noise is trace-class (Hilbert-Schmidt), ensuring the SPDE \eqref{SPDE} is well-posed.
Here, $X$ is a separable Hilbert space whose dual space
is identified with $X$ itself. We assume the decomposition
\begin{equation*}
\mathcal{L}(u)=\mathcal{A}u+\mathcal{F}(u),
\end{equation*}
where the unbounded linear operator $\mathcal{A}$ is the generator of a strongly continuous semigroup on a Hilbert space $\mathcal{H}$, and the nonlinear operator $\mathcal{F}: \mathcal{H}\rightarrow\mathcal{H}$ encodes the
nonlinearity. Rewriting \eqref{SPDE} as an
abstract stochastic evolution equation
\begin{equation*}
du(t)=[\mathcal{A}u(t)+\mathcal{F}(u(t))]dt+G(u(t))dW(t).
\end{equation*}
Since the solutions of the above equation are typically insufficiently regular for classical interpretations of the partial differential equation, they are understood in a weaker sense. One possibility
is the mild formulation, derived from the variation of constants formula.

\begin{definition} (Existence and uniqueness of local mild solution).
Let $(\Omega,\mathcal{F},\mathcal{F}_t,\mathbb{P})$ be a filtered probability space and $\mathcal{H}$ be a phase space. A continuous stochastic process $u\in C^{0}([0,\tau^{\ast}),X)$, defined for some almost surely positive
stopping time $\tau^{\ast}>0$, is called a mild solution of \eqref{SPDE} with initial condition $u_0$ if the following equality
\begin{equation}\label{I}
u(t)=e^{t\mathcal{A}}u_0+\int_{0}^{t}e^{(t-s)\mathcal{A}}\mathcal{F}(u(s))ds+\int_{0}^{t}e^{(t-s)\mathcal{A}}G(u(s))dW(s)
\end{equation}
holds $\mathbb{P}$-almost surely for all $t\in(0,\tau^{\ast})$. We say that pathwise uniqueness holds if, for any two
continuous processes $u_1$ and $u_2$ defined on the same probability space and satisfying \eqref{I}, we have
\begin{equation*}
\mathbb{P}\big(u_1(t)=u_2(t)\,\,\text{for all}\,\,t\in(0,\tau^{\ast})\big)=1.
\end{equation*}
\end{definition}

\begin{remark}
If $G$ is the identity operator on $\mathcal{H}$, the SPDEs in \eqref{SPDE} are driven by additive white noise.
Otherwise, they are perturbed by multiplicative white noise.
\end{remark}
We assume that the local mild solution to the SPDE \eqref{SPDE} is unique and corresponds to a stochastic
process with sufficient regularity in $\mathcal{H}$. Furthermore, we aim to approximate this solution utilizing ansatz functions of the form
$u^{h}(t,x):=u(t,x;h)$, where the time-dependent coordinate $h$ parameterizes the position on the approximate slow manifold.

\begin{definition} (Approximate slow manifold).
For an open parameter space $\mathcal{P}\subset\mathbb{R}^{n}$ (e.g., position of interfaces), we consider an $n$-dimensional manifold
\begin{equation*}
\mathcal{M}=\{u^{h}\in\mathcal{H}:h\in\mathcal{P}\}
\end{equation*}
embedded in an infinite-dimensional stochastic system within a  Hilbert space $\mathcal{H}$, equipped with the
scalar product $\langle\cdot,\cdot\rangle$ and induced corresponding norm $\|\cdot\|$. We assume that $\mathcal{M}$ is non-degenerate.
Furthermore, we suppose that the map $h\mapsto u^{h}$ defines a $C^3$-parametrization of $\mathcal{M}$. We denote the $j$-th
partial derivative of $u^h$ with respect to $h_j$ by $u_j^h:=\partial_{h_{j}}u^{h}$, $j=1,\cdot\cdot\cdot,n$.
\end{definition}

\begin{definition} (Fermi coordinates).
As long as the solution $u$ stays close to $\mathcal{M}$,
we decompose $u$ into a point $u^{h}$ on the manifold and the component orthogonal
to the manifold. To be more precise, we define the pair of Fermi coordinates (or tubular coordinates) $(h,v)\in\mathcal{P}\times\mathcal{H}$ for $u$,
such that
\begin{equation*}
u=u^{h}+v\quad\text{with}\quad v\bot\mathcal{M}.
\end{equation*}
If the manifold is not degenerate, the derivatives $u_{j}^{h}$
span the entire tangent space. Consequently, the orthogonality condition $v\perp \mathcal{M}$ is equivalent to
\begin{equation*}
\langle u_{j}^{h},v\rangle=0\quad\text{for all}\quad j=1,\cdot\cdot\cdot,n.
\end{equation*}
\end{definition}

\begin{lemma}\label{Q}
Let $W$ be a $Q$-Wiener process in the underlying Hilbert space $\mathcal{H}$. For any $u,v\in\mathcal{H}$, the quadratic covariation of the stochastic integrals satisfies
\begin{equation*}
\langle u,dW\rangle\langle v,dW\rangle=\langle u,Qv\rangle dt.
\end{equation*}
\end{lemma}
\begin{proof}
We express the $Q$-Wiener process as $W(t)=\sum_{k\in\mathbb{N}}\alpha_kB_ke_k$, where the set of the eigenfunctions $\{e_k\}_{k\in\mathbb{N}}$ forms a complete $\mathcal{H}$-orthonormal basis, and $\{B_k\}_{k\in\mathbb{N}}$ are identically distributed $\mathbb{R}$-valued scalar Brownian motions. The symmetric operator $Q$ satisfies $Qe_{k}=\alpha_k^{2}e_k$, where $\alpha_k^{2}$ are its eigenvalues. Using the It\^{o} differential rule $dB_kdB_l=\delta_{k,l}dt$, we compute
\begin{equation*}
\langle u,dW\rangle\langle v,dW\rangle=\sum_{k,l}\alpha_k\alpha_l\langle u,e_k\rangle\langle v,e_l\rangle dB_kdB_l=\sum_{k}\alpha_k^{2}\langle u,e_k\rangle\langle v,e_k\rangle dt.
\end{equation*}
By Parseval's identity, this yields the claimed result. Only the quadratic terms ($(dB_k)^2=dt$) contribute to the differential, as cross-terms ($dB_kdB_l$ for $k\neq l$) vanish.
\end{proof}

A crucial result is the following theorem.

\begin{theorem}\label{Equivalence}
(Equivalence of full SPDE and coupled system)
Suppose $u=u^{h}+v$, where $h\in \mathcal{P}$ and $v\bot \mathcal{M}$. Then
$u$ is a solution of \eqref{SPDE} if and only if $h$ solves the stochastic differential equation
\begin{equation}\label{bW}
dh_{j}=b_j(h)dt+\langle\sigma_{j}(h),dW\rangle,
\end{equation}
where $\sigma_j$ and $b_j$ are determined by \eqref{A}, and $v$ satisfies the evolution equation
\begin{equation}\label{V}
dv=\mathcal{L}(u^{h}+v)dt+G(u^{h}+v)dW-\sum_{k}u_{k}^{h}dh_k-\frac{1}{2}\sum_{k,l} u_{kl}^{h}\langle\sigma_k,Q\sigma_l\rangle dt.
\end{equation}
\end{theorem}
\begin{proof}
Given the decomposition $u=u^{h}+v$ with $h\in\mathcal{P}$ and $v\bot\mathcal{M}$, if $u$ is a solution of \eqref{SPDE},
our goal is to determine
\begin{equation*}
b_{j}:\mathcal{P}\mapsto\mathbb{R}^{n}\quad\text{and}\quad\sigma_{j}:\mathcal{P}\mapsto\mathcal{H},
\end{equation*}
such that $u(t)\approx u^{h(t)}$, where the diffusion process $h$ solves the system of SDEs \eqref{bW}.

Assume the solution $u$ to \eqref{SPDE} can be decomposed into Fermi coordinates with $h$ governed by \eqref{bW}. Differentiating
\begin{equation*}
u=u^{h}+v\quad\quad\text{and}\quad\quad\langle u_{j}^{h},v\rangle=0,
\end{equation*}
we derive via It\^{o}'s formula and It\^{o}'s product rule:
\begin{equation}\label{uv}
\sum_{k}u_{k}^{h}dh_{k}+\underset{\text{It\^{o}-correction}}{\underbrace{\frac{1}{2}\sum_{k,l}u_{kl}^{h}dh_{k}dh_{l}}}+dv=du=
\mathcal{L}(u^{h}+v)dt+G(u^{h}+v)dW,
\end{equation}
and
\begin{equation}\label{product}
0=\langle du_{j}^{h},v\rangle+\langle u_{j}^{h},dv\rangle+\underset{\text{It\^{o}-correction}}{\underbrace{\langle du_{j}^{h},dv\rangle}}.
\end{equation}

Now we calculate $\langle du_{j}^{h},dv\rangle$. Using Lemma \ref{Q} and the fact that $dWdt=0$ and $dtdt=0$,
\begin{align}\nonumber
\langle du_{j}^{h},dv\rangle&=\sum_{k}\langle u_{jk}^{h}dh_k,G(u^{h}+v)dW\rangle-\sum_{k}\langle u_{jk}^{h}dh_k,\sum_{l}u_{l}^{h}dh_l\rangle \\ \nonumber
   &=\sum_{k}\langle\sigma_k,dW\rangle\langle u_{jk}^{h},G(u^{h}+v)dW\rangle-\sum_{k,l}\langle u_{jk}^{h},u_{l}^{h}\rangle\langle\sigma_k,dW\rangle\langle\sigma_l,dW\rangle\\ \label{dudv}
   &=\sum_{k}\langle G(u^{h}+v)u_{jk}^{h},Q\sigma_k\rangle dt-\sum_{k,l}\langle u_{jk}^{h},u_{l}^{h}\rangle
   \langle \sigma_k,Q\sigma_l\rangle dt:=F_j(u^h,v)dt.
\end{align}
From
\begin{equation}\label{dh}
dh_kdh_l=\langle \sigma_k,Q\sigma_l\rangle dt
\end{equation}
substitute into \eqref{uv} to get the expression of $dv$ in \eqref{V}, and also expand
\begin{align}\nonumber
\langle du_{j}^{h},v\rangle&=\sum_{k}\langle u_{jk}^{h},v\rangle dh_k+\frac{1}{2}\sum_{k,l}\langle u_{jkl}^{h},v\rangle dh_kdh_l  \\ \label{vdu}
                             &=\sum_{k}\langle u_{jk}^{h},v\rangle dh_k+\frac{1}{2}\sum_{k,l}\langle u_{jkl}^{h},v\rangle\langle \sigma_k,Q\sigma_l\rangle dt.
\end{align}
Substituting \eqref{dudv} and \eqref{vdu} into \eqref{product}, and expressing $\langle u_{j}^{h},dv\rangle$ via \eqref{uv}, we obtain
\begin{align*}
  0&\overset{\eqref{product}}{=}\langle du_{j}^{h},v\rangle+\langle u_{j}^{h},dv\rangle+\langle du_{j}^{h},dv\rangle \\
   &\underset{\eqref{vdu}}{\overset{\eqref{dudv}}{=}}\sum_{k}\langle u_{jk}^{h},v\rangle dh_{k}+\frac{1}{2}\sum_{k,l}\langle u_{jkl}^{h},v\rangle\langle\sigma_k,Q\sigma_l\rangle dt+\langle u_{j}^{h},dv\rangle+F_j(u^h,v)dt\\
   &\overset{\eqref{uv}}{=}\sum_{k}\langle u_{jk}^{h},v\rangle dh_{k}+\frac{1}{2}\sum_{k,l}\langle u_{jkl}^{h},v\rangle\langle\sigma_k,Q\sigma_l\rangle dt+\langle u_{j}^{h},\mathcal{L}(u^{h}+v)\rangle dt
   +\langle u_{j}^{h},G(u^{h}+v)dW\rangle \\&\quad\quad-\sum_{k}\langle u_{j}^{h},u_{k}^{h}\rangle dh_{k}-\frac{1}{2}\sum_{k,l}\langle u_{j}^{h},u_{kl}^{h}\rangle\langle\sigma_k,Q\sigma_l\rangle dt+F_j(u^h,v)dt.
\end{align*}
Define the matrix  $A$ with entries
\begin{equation}\label{Ajk}
A_{j,k}:=\langle u_{j}^{h},u_{k}^{h}\rangle-\langle u_{jk}^{h},v\rangle.
\end{equation}
  Finally,
\begin{align} \label{A}
\sum_{k}A_{j,k}dh_k&=\sum_{k}\big(\langle u_{j}^{h},u_{k}^{h}\rangle-\langle u_{jk}^{h},v\rangle\big)dh_k\\ \nonumber
&=\langle u_{j}^{h},\mathcal{L}(u^{h}+v)\rangle dt+\langle u_{j}^{h},G(u^{h}+v)dW\rangle +\frac{1}{2}\sum_{k,l}\big(\langle u_{jkl}^{h},v\rangle-\langle u_{j}^{h},u_{kl}^{h}\rangle\big)\langle\sigma_k,Q\sigma_l\rangle dt+F_j(u^h,v)dt.
\end{align}
We denote the left-hand side of \eqref{A} as $A(h,v)\cdot dh$. If the matrix $A(h,v)$ is invertible, we determine first $\sigma_j$ as prefactor of $dW$ and then $b_j$.
 Then,
the drift $b$ and diffusion $\sigma$ are given by
\begin{equation}\label{bh}
b_{j}(h)=\sum_{k}A_{jk}^{-1}\Big[\langle u_{j}^{h},\mathcal{L}(u^{h}+v)\rangle+\frac{1}{2}\sum_{l}\big(\langle u_{jkl}^{h},v\rangle-\langle u_{j}^{h},u_{kl}^{h}\rangle\big)\langle\sigma_k,Q\sigma_l\rangle+F_j(u^h,v)\Big]
\end{equation}
and
\begin{equation}\label{sigmah}
\sigma_{j}(h)=\sum_{k}A_{jk}^{-1}G(u^{h}+v)^{*} u_{j}^{h},
\end{equation}
respectively. In the last equation, $G(u^{h}+v)^{*}$ is the adjoint operator of $G(u^{h}+v)$. By substituting \eqref{bh}-\eqref{sigmah} into \eqref{bW}, it can be concluded that $h$ is indeed a semimartingale, composed of a bounded-variation drift term and a martingale.

For the proof, note that we established one direction above. We now consider the converse.

Suppose the pair of functions $(h,v)$  is the solution to the system governed by $\eqref{bW}$ and $\eqref{V}$, where $b$ and $\sigma$
are characterized by $\eqref{bh}$ and $\eqref{sigmah}$, respectively. Assume the matrix $A(h,v)$ given by $\eqref{Ajk}$ is invertible for all time $t$, and the initial condition $u(0)=u^{h(0)}+v(0)$ satisfies $\langle u_{j}^{h(0)},v(0)\rangle=0$ for all $j=1,\cdot\cdot\cdot,n$. Then $u=u^{h}+v$ solves $\eqref{SPDE}$ with $\langle u_{j}^{h},v\rangle=0$ for $j=1,\cdot\cdot\cdot,n$. This follows by reversing the earlier arguments.

The orthogonality $\langle u_{j}^{h},v\rangle=0$ is preserved because $d\langle u_{j}^{h},v\rangle=0$ for all $j\in\{1,\cdot\cdot\cdot,n\}$ and $\langle u_{j}^{h(0)},v(0)\rangle=0$. To verify this, we calculate
\begin{align*}
d\langle u_{j}^{h},v\rangle=&\langle du_{j}^{h},v\rangle+\langle u_{j}^{h},dv\rangle+\langle du_{j}^{h},dv\rangle \\
                           =&\langle du_{j}^{h},v\rangle+\langle u_{j}^{h},dv\rangle+\langle du_{j}^{h},du\rangle-\langle du_{j}^{h},du^{h}\rangle\\
                           =&\sum_{k}\langle u_{jk}^{h},v\rangle dh_k+\frac{1}{2}\sum_{k,l}\langle u_{jkl}^{h},v\rangle\langle \sigma_k,Q\sigma_l\rangle dt+\langle u_{j}^{h},\mathcal{L}(u^{h}+v)\rangle dt
   +\langle u_{j}^{h},G(u^{h}+v)dW\rangle\\
   &-\sum_{k}\langle u_{j}^{h},u_{k}^{h}\rangle dh_k
   -\frac{1}{2}\sum_{k,l}\langle u_{j}^{h},u_{kl}^{h}\rangle\langle\sigma_k,Q\sigma_l\rangle dt+\sum_{k}\langle G(u^{h}+v)u_{jk}^{h},Q\sigma_k\rangle dt-\sum_{k,l}\langle u_{jk}^{h},u_{l}^{h}\rangle
   \langle \sigma_k,Q\sigma_l\rangle dt.
\end{align*}
We first collect the $dW$-terms, and then
\begin{equation*}
G(u^{h}+v)^{*} u_{j}^{h}-\sum_{k}\big(\langle u_{j}^{h},u_{k}^{h}\rangle-\langle u_{jk}^{h},v\rangle\big)\sigma_k\overset{\eqref{Ajk}}{=}G(u^{h}+v)^{*} u_{j}^{h}-\sum_{k}A_{j,k}\sigma_k\overset{\eqref{sigmah}}{=}0.
\end{equation*}
For the remaining  $dt$-terms we have
\begin{align*}
   &\frac{1}{2}\sum_{k,l}\langle u_{jkl}^{h},v\rangle\langle \sigma_k,Q\sigma_l\rangle +\langle u_{j}^{h},\mathcal{L}(u^{h}+v)\rangle -\frac{1}{2}\sum_{k,l}\langle u_{j}^{h},u_{kl}^{h}\rangle\langle\sigma_k,Q\sigma_l\rangle   \\
   &+\sum_{k}\langle G(u^{h}+v)u_{jk}^{h},Q\sigma_k\rangle -\sum_{k,l}\langle u_{jk}^{h},u_{l}^{h}\rangle
   \langle \sigma_k,Q\sigma_l\rangle -\sum_{k}\big(\langle u_{j}^{h},u_{k}^{h}\rangle-\langle u_{jk}^{h},v\rangle\big)b_k\\
   &=\langle u_{j}^{h},\mathcal{L}(u^{h}+v)\rangle+\frac{1}{2}\sum_{l}\big(\langle u_{jkl}^{h},v\rangle-\langle u_{j}^{h},u_{kl}^{h}\rangle\big)\langle\sigma_k,Q\sigma_l\rangle+F_j(u^h,v)-\sum_{k}A_{j,k}b_k\overset{\eqref{bh}}{=}0.
\end{align*}
Thus, $d\langle u_{j}^{h},v\rangle=0$,  and the orthogonality $\langle u_{j}^{h},v\rangle=0$ holds.
\end{proof}

\begin{remark}
 The coupled system \eqref{bW}-\eqref{V} is equivalent to the full SPDE \eqref{SPDE}. We can use \eqref{V} to study the stability of $\mathcal{M}$ and \eqref{bW} to analyze the motion along $\mathcal{M}$. To simplify the equation \eqref{bW}, we may neglect all terms that depend on $v$, assuming proximity to the manifold.
\end{remark}

We assumed that the matrix $A$ is invertible to solve the equation \eqref{A} for $dh$, thereby obtaining the exact stochastic formula for the shape variable $h$ given by the equation \eqref{bW}. We shall now proceed to put the assumption to the following corollary for verification.
\begin{corollary}\label{vA}
If $v$ is sufficiently small, then the  matrix $A$ is invertible.
\end{corollary}
\begin{proof}
If the solution $u$ is sufficiently close to the slow manifold $\mathcal{M}$, there exists a unique $h\in\mathcal{P}$ such that
\begin{equation*}
\text{dist}(u,\mathcal{M})=\inf_{\hat{h}\in\mathcal{P}}\|u-u^{\hat{h}}\|=\|u-u^{h}\|.
\end{equation*}
Differentiating the map $\Phi(h)=\|u-u^{h}\|^2/2$ with
respect to all variables $h_j$, we obtain
\begin{equation*}
\langle u-u^{h}, u_j^{h}\rangle=0\quad\text{for all}\quad j=1,\cdot\cdot\cdot,n.
\end{equation*}
This implies that the vector $v=u-u^{h}$ is orthogonal to the tangent space $\mathcal{T}_{u^{h}}\mathcal{M}$ of $\mathcal{M}$
in $u^{h}$. The Hessian matrix of the map $\Phi$,
\begin{equation*}
(H_{\Phi}(h))_{j,k}=\langle u_{j}^{h},u_{k}^{h}\rangle-\langle u_{jk}^{h},v\rangle,
\end{equation*}
is closely related to the first fundamental form $P$ of the manifold $\mathcal{M}$, defined by $P_{jk}=\langle u_{j}^{h},u_{k}^{h}\rangle$. Since the vectors $u_{j}^{h}$, $j=1,\cdot\cdot\cdot,n$, form a basis of the tangent space $\mathcal{T}_{u^{h}}\mathcal{M}$, for any $w\in\mathcal{T}_{u^{h}}\mathcal{M}$ we can find a vector $\beta=(\beta_1,\cdot\cdot\cdot,\beta_n)$
such that $w=\sum_{j}\beta_ju_{j}^{h}$. Consequently,
\begin{equation*}
\|w\|^{2}=\sum_{j,k=1}^{n}\beta_j\beta_k\langle u_{j}^{h},u_{k}^{h}\rangle=\beta^{\top}P\beta,
\end{equation*}
which shows that the fundamental form $P$ is positive definite. If the distance $\|u-u^{h}\|$ is sufficiently
small, the Hessian $H_{\Phi}(h)$ is also positive definite, ensuring that $\Phi$ attains
minima only near the slow manifold $\mathcal{M}$.
\end{proof}
\begin{remark}
The invertibility of the matrix $A$ is not limited to $v$ having a sufficiently small norm. Furthermore, the converse of Corollary \ref{vA} does not hold.
\end{remark}

\section{Effective dynamics of stochastic models}\label{ED}
In this section, we provide four examples about stochastic wave models
to corroborate our analytical results of effective models.

\subsection{Stochastic damped wave equation with multiplicative noise}\label{DW}
The stochastic damped wave equation  under multiplicative noise models oscillatory systems with dissipation and noise, such as vibrating membranes under random forcing \cite{T}. Now we combine dimension reduction techniques to quantify solutions near a slow manifold in the phase space.

Consider the stochastic damped wave equation driven by multiplicative noise on $\mathcal{D}=[0,1]$ with Dirichlet boundary conditions:
\begin{equation}\label{sdw}
\partial_t^{2}u+\gamma\partial_tu=\Delta u-u^{3}+\varepsilon u\circ\partial_tW(t,x),\quad u(0,x)=u_0(x),\, \partial_tu(0,x)=w_0(x),
\end{equation}
where $\gamma>0$ is the damping coefficient, $\varepsilon$ controls the noise intensity, $W(t,x)$ is space-time white noise, and $\circ$ denotes Stratonovich integration for state-dependent noise. Rewrite \eqref{sdw} as a first-order system
\begin{equation*}\left\{
\begin{array}{l}
du=wdt,\\
dw=\Big(-\gamma w+\Delta u-u^{3}\Big)dt+\varepsilon u\circ dW(t,x).
\end{array}\right.
\end{equation*}
For small noise ($0<\varepsilon\ll1$), we decompose $u(t,x)=u^{h(t)}(x)+v(t,x)$, where $u^{h(t)}(x)=h(t)\sin(\pi x)\in\mathcal{M}$ and $v\bot\mathcal{M}$.
The parameter $h(t)$ satisfies
\begin{equation*}
dh=\Big[-\frac{\pi^2}{\gamma}h-\frac{3}{\gamma}h^3\Big]dt+\varepsilon h\circ dB(t),
\end{equation*}
where $B(t)$ is a scalar Brownian motion derived from projecting $W(t,x)$ onto the manifold basis with the function $\sin(\pi x)$.
It yields
\begin{equation*}
B(t)=\int_0^{t}\int_{\mathcal{D}}\sin(\pi x)dW(s,x)
\end{equation*}

The diffusion coefficient $\varepsilon h$ depends on $h$, reflecting state-dependent noise. The Stratonovich formulation preserves the chain rule, avoiding It\^{o}-to-Stratonovich corrections. Converting the above reduced SDE to  the It\^{o} form
\begin{equation*}
dh=\Big(-\frac{\pi^2}{\gamma}h-\frac{3}{\gamma}h^{3}+\frac{\sigma^{2}}{2}h\Big)dt+\varepsilon hdB(t).
\end{equation*}
The deterministic flow
\begin{equation*}
\partial_tu=\frac{1}{\gamma}(\Delta u-u^3)
\end{equation*}
has a spectral gap $\lambda=\pi^2/\gamma>0$. This ensures the exponential stability. We assume the non-degeneracy $\|u^{h}\|\geq c>0$, ensuring $\varepsilon h$ does not vanish. The orthogonal component $v(t,x)$ satisfies
\begin{equation*}
\|v(t)\|_{H^1}\leq e^{-\lambda t}\|v_0\|_{H^1}+\varepsilon\int_{0}^{t}e^{-\lambda(t-s)}\|u^{h}(s)\|_{L^2}dW(s).
\end{equation*}
Hence, the term $v(t)$ is bounded via
\begin{equation*}
\mathbb{E}(\sup_{0<t\leq T}\|v(t)\|_{H^1}^{2})\leq C\left(e^{-2\lambda T}\|v_0\|_{H^1}^{2}+\varepsilon^{2}\int_0^{T}e^{-2\lambda(T-s)}\mathbb{E}(\|u^{h}(s)\|_{L^2}^2)ds\right).
\end{equation*}

Fig. \ref{wave} simulates the reduced one-dimensional SDE for the stochastic damped wave equation with multiplicative noise and generates plots illustrating the dynamics. The deterministic solution decays to $h=0$, while the stochastic solution fluctuates around $h=0$ with variance controlled by $\varepsilon$ due to multiplicative noise but remains near the manifold, as seen in the time series.
The distribution of $h(T)$ is symmetric and centered at $h=0$, as shown in the histogram. The deterministic drift (black curve) pulls $h$ toward 0, while stochastic paths (red) explore noise-driven deviations.

\begin{figure}
\begin{center}
\begin{minipage}{3.2in}
\leftline{(a)}
\includegraphics[width=3.2in]{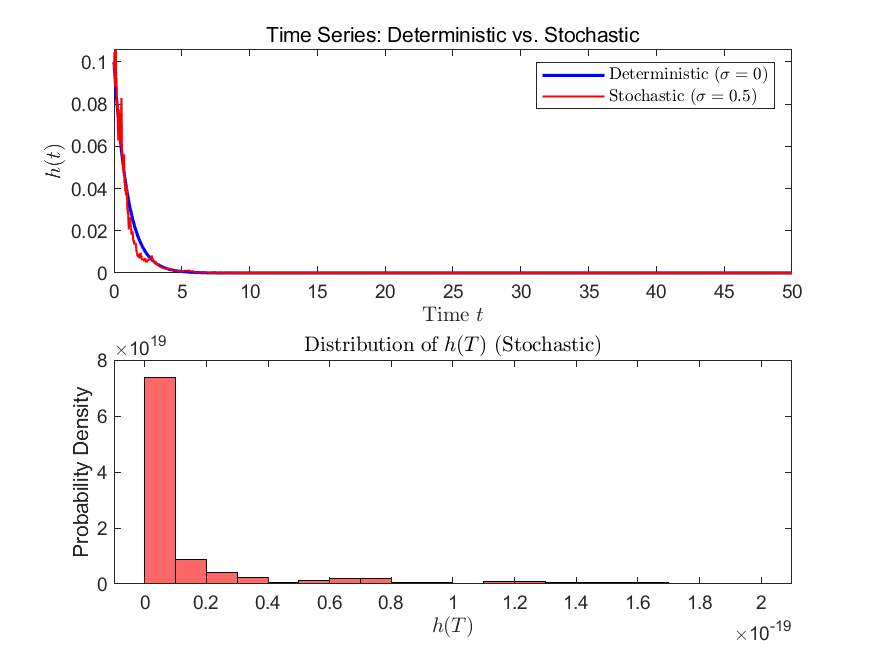}
\end{minipage}
\hfill
  \begin{minipage}{3.2in}
\leftline{(b)}
\includegraphics[width=3.2in]{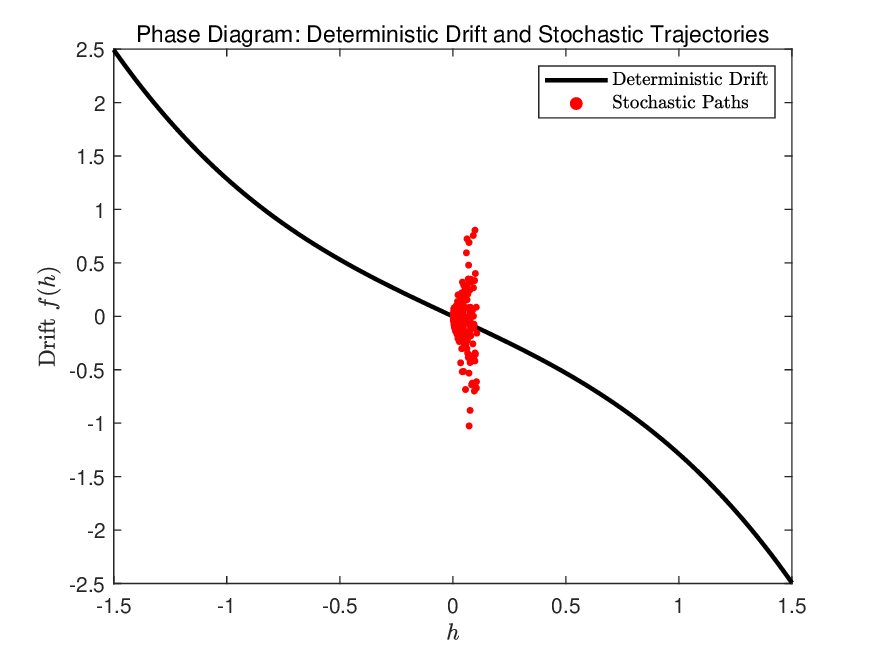}
\end{minipage}
\caption{(a) A time series comparison (deterministic vs. stochastic), and a  histogram of the final state $h(T)$; (b) A phase diagram showing the deterministic drift and sample paths. Damping coefficient $\gamma=10$; Linear damping rate $a=\pi^2/\gamma$; Nonlinear damping coefficient $b=3/\gamma$; Noise intensity $\varepsilon=0.5$; Initial condition $h_0 = 0.1$; Total simulation time $T=50$; Time step $dt=0.01$.}\label{wave}
\end{center}
\end{figure}
\subsection{Stochastic Allen-Cahn  equation with multiplicative noise}\label{ACE}
The Allen-Cahn equation models phase separation in bistable systems, with applications in material science and biology \cite{HLR}. Its deterministic dynamics exhibit front solutions connecting stable equilibria. Under weak noise, these fronts fluctuate near a slow manifold, with transitions between metastable states \cite{W}. Now we quantify the stochastic stability of front solutions and characterize exit times using dimension reduction.

Consider the stochastic Allen-Cahn equation with Neumann boundary conditions $\partial_{x}u|_{x=0}=\partial_{x}u|_{x=L}=0$:
\begin{equation*}
\partial_tu=\Delta u+u-u^{3}+\varepsilon u\circ \partial_tW(t,x), \quad  x\in[0,L],
\end{equation*}
where
$u(t,x)$ represents the order parameter (e.g., phase separation), $\varepsilon>0$ scales the multiplicative Stratonovich noise,
and $W(t,x)$ is a space-time white noise. This SPDE models phase separation with stochastic fluctuations.

For weak noise ($\varepsilon\ll1$), we decompose the solution as $u(t,x)=u^{h(t)}(x)+v(t,x)$, where $v\bot\mathcal{M}$. For
$L\gg1$, a front located near position $h(t)$ can be approximated as
\begin{equation*}
u^{h}(x)=\tanh\Big(\frac{x-h(t)}{\sqrt{2}}\Big).
\end{equation*}
The kink solutions (heteroclinic fronts) parameterized by their interface position $h$ form a slow manifold $\mathcal{M}=\{u^{h}(x)\in\mathcal{H}: h\in[0,L]\}$, where $\mathcal{H}=L^{2}(\mathbb{R})$. The tangent space $T_{u^h}\mathcal{M}$ at $u^h$ is spanned by $\partial_hu^h$, which represents the translational mode of the front.
The orthogonal component $v$ satisfying $\langle\partial_hu^h,v\rangle=0$ captures fluctuations away from the slow manifold.

The tangent vector to $\mathcal{M}$ (Goldstone mode) is
\begin{equation*}
\psi_h(x):=\partial_hu^h(x)=-\frac{1}{\sqrt{2}}\text{sech}^{2}(\frac{x-h}{\sqrt{2}}).
\end{equation*}

Using Theorem \ref{Equivalence}, we project the SPDE onto the tangent space $T_{u^h}\mathcal{M}$.
The matrix $A$ (scalar here) is
\begin{equation*}
A=\langle\psi_h,\psi_h\rangle-\langle\partial_h\psi_h,v\rangle\approx\langle\psi_h,\psi_h\rangle=\|\psi_h\|^{2},
\end{equation*}
since $v$ is small. The term $\langle\partial_h\psi_h,v\rangle$ is negligible under the approximation $v\approx0$.
The drift is
\begin{equation*}
b(h)=\frac{\langle\psi_h,\mathcal{L}(u^{h})\rangle}{\|\psi_h\|^{2}}\approx e^{-\sqrt{2}L},\quad \text{where}\,\, \mathcal{L}(u^{h})=\Delta u^h+u^h-(u^h)^{3}.
\end{equation*}
It is the exponentially small drift from neighboring fronts.
We compute  the diffusion
\begin{equation*}
\sigma(h)=\frac{\langle\psi_h,G(u^{h})\rangle}{\|\psi_h\|^{2}}=\frac{\varepsilon}{\|\psi_h\|^{2}}\int_0^{L}\psi_h(x)u^h(x)dx,\quad \text{where}\,\, G(u^{h})=\varepsilon u^{h}.
\end{equation*}
The reduced dynamics for the front position $h(t)$ become
\begin{equation*}
dh=b(h)dt+\sigma(h)\circ dB(t),
\end{equation*}
where $B(t)$ is a scalar Brownian motion.

Fig. \ref{AC} simulates the reduced SDE for $h(t)$ for $L=20$.   The front position $h(t)$ fluctuates near its initial position $h_0$ for exponentially long times, which evolves at an exponentially slowly rate due to metastability.

\begin{figure}
\centering
\includegraphics[width=10cm]{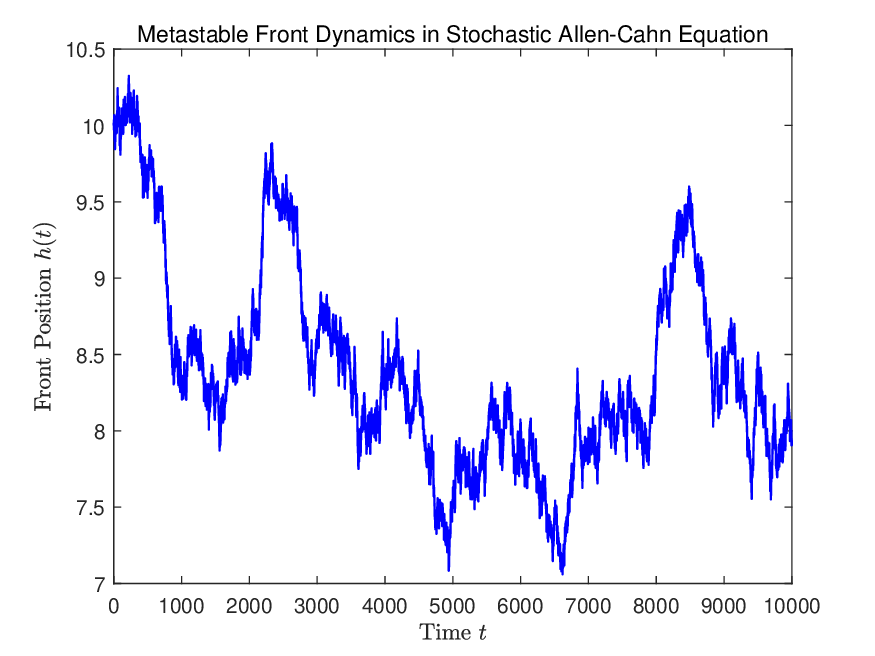}
\caption{Front position $h(t)$ in stochastic Allen-Cahn equation. Domain length $L=20$; Noise intensity $\varepsilon=0.1$; Initial front position $h_0=L/2$; Total time $T=10^4$; Time step $dt=0.1$.
}\label{AC}	
\end{figure}

\subsection{Stochastic nonlinear Schr\"{o}dinger equation with multiplicative noise}\label{SE}
The nonlinear Schr\"{o}dinger equation governs wave propagation in diverse media, from optical fibers to Bose-Einstein condensates. Its soliton solutions, arising from a balance between dispersion and nonlinearity, are stable under deterministic dynamics. However, real-world systems often experience stochastic perturbations, necessitating an understanding of noise effects on soliton stability \cite{BD}.

We consider the stochastic nonlinear Schr\"{o}dinger equation with multiplicative Stratonovich space-time white noise on a spatial domain $x\in\mathbb{R}$:
\begin{equation*}
i du=[-\partial_x^{2}u-|u|^{2}u]dt+i\varepsilon u\circ dW(t,x),
\end{equation*}
where $u(t,x)$ is a complex-valued wave function, $\varepsilon>0$ scales the noise intensity, and $W(t,x)$ is a space-time white noise process. This equation models wave propagation in nonlinear media with stochastic perturbations (e.g., optical fibers with random gain or loss). A critical question is whether solitons maintain coherence under noise, or if their positions diffuse unpredictably.

For weak noise ($\varepsilon\ll1$), we decompose the solution as $u(t,x)=u^{h(t)}(x)+v(t,x)$, where $v\bot T_{u^{h}}\mathcal{M}$. The soliton position $h(t)$ is a free parameter due to translational invariance, forming a slow manifold $\mathcal{M}=\{u^h(x): h\in\mathbb{R}\}$ with
\begin{equation*}
u^{h(t)}(x)=\eta(x-h(t))e^{i\phi(t)},
\end{equation*}
where $\eta(x)=\sqrt{2}\text{sech}(x)$ is the soliton profile, and $\phi(t)$ is the phase. The tangent space $T_{u^{h}}\mathcal{M}$ is spanned by the Goldstone mode $\psi_h(x)=-\eta'(x-h)e^{i\phi}$ (tangent to $\mathcal{M}$), arising from translational symmetry.

Projecting the SPDE onto
$\mathcal{M}$ yields the reduced dynamics for $h(t)$. The projection gives the Stratonovich SDE
\begin{equation}\label{SNLS}
dh=\varepsilon\sqrt{\langle\psi_h,\psi_h\rangle}\circ dB(t),\quad \psi_h(x)=\partial_h u^h(x),
\end{equation}
where $B(t)$ is a scalar Brownian motion. Translational symmetry nullifies the drift, leaving only noise-driven dynamics. Converting to  It\^{o} form
introduces a noise-induced drift
\begin{equation}\label{INLS}
dh=\frac{\varepsilon^2}{2}\langle\psi_h,\psi_h\rangle dt+\varepsilon\sqrt{\langle\psi_h,\psi_h\rangle}dB(t).
\end{equation}
Here, $\langle\psi_h,\psi_h\rangle=\|\eta'\|^{2}_{L^{2}}=\frac{4}{3}$ is constant, computed from $\eta(x)=\sqrt{2}\text{sech}(x)$.
The linearized operator around $u^h$ has a spectral gap $\lambda\sim1$, ensuring exponential decay of orthogonal perturbations.
Solutions remain close to $\mathcal{M}$ for exponential long times, demonstrating remarkable coherence.

\begin{figure}
\centering
\includegraphics[width=10cm]{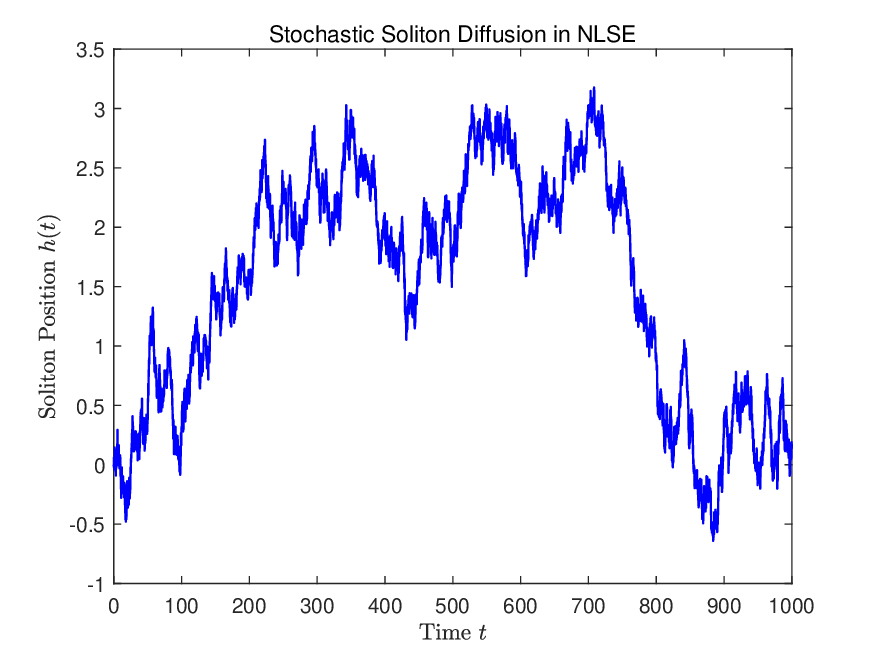}
\caption{Simulated trajectory of $h(t)$ from the reduced SDE, showing noise-driven diffusion with drift. The soliton maintains positional coherence over long times despite stochastic perturbations. Noise intensity $\varepsilon=0.1$; Initial position $h_0=0$;  Total time $T=10^3$; Time step $dt=0.1$.
}\label{NLSE}	
\end{figure}

We simulate the It\^{o} SDE \eqref{INLS} and visualize diffusive motion numerically in Fig. \ref{NLSE}. The soliton position $h(t)$ exhibits Brownian motion with a drift proportional to $\varepsilon^2$, consistent with \eqref{INLS}, reflecting noise-induced translational diffusion. Crucially,  the soliton profile remains coherent (shape-preserving) over exponentially long timescales, as orthogonal modes $v(t,x)$ are suppressed by the spectral gap.

\subsection{Stochastic Swift-Hohenberg equation with multiplicative noise}\label{SHE}
The stochastic Swift-Hohenberg equation models pattern formation (e.g., stripes, hexagons) in systems such as Rayleigh-B\'enard convection and morphogenesis, with noise perturbing the amplitude of patterns \cite{HO}. Its captures the interplay between deterministic symmetry-breaking and random fluctuations \cite{KBM}.

We study the stochastic Swift-Hohenberg equation incorporating multiplicative Stratonovich noise on a periodic spatial domain $x\in[0,L]$:
\begin{equation*}
\partial_tu=-(\partial_x^{2}+1)^{2}u+\delta u-u^3+\varepsilon u\circ\partial_tW(t,x),
\end{equation*}
where $\delta>0$ is the bifurcation parameter, $\varepsilon>0$ scales the  noise,
and $W(t,x)$ is a space-time white noise.

For $\delta\ll1$, we decompose the solutions as $u(t,x)=u^{h(t)}(x)+v(t,x)$, where $u^{h(t)}(x)$ with amplitude and phase lies on the slow manifold of stripe patterns
\begin{equation*}
\mathcal{M}=\{u^h(x)=h\cos(x-\phi): h\in\mathbb{R}, \phi\in[0,2\pi)\},
\end{equation*}
and $v$ is orthogonal to $\mathcal{M}$. The tangent space $T_{u^h}\mathcal{M}$ is spanned by $\partial_{h}u^h=\cos(x-\phi)$ and $\partial_{\phi}u^h=h\sin(x-\phi)$.
Weak noise perturbs amplitude dynamics, inducing stochastic fluctuations in $h(t)$.

Projecting the SPDE onto $\partial_hu^{h}$ yields the Stratonovich SDE
\begin{equation*}
dh=(\delta h-3h^{3})dt+\varepsilon h\circ dB(t),
\end{equation*}
where $B(t)$ is a scalar Brownian motion derived from projecting $W(t,x)$ onto the $\cos(x-\phi)$ mode.
Converting to It\^{o} form introduces a noise-induced drift
\begin{equation}\label{ISH}
dh=(\delta h-3h^{3}+\frac{\varepsilon^{2}}{2}h)dt+\varepsilon hdB(t).
\end{equation}
The equilibrium amplitude shifts slightly to $h_{\ast}=\sqrt{(\delta+\varepsilon^2/2)/3}$, but for $\varepsilon^2\ll\delta$, $h_{\ast}\approx\sqrt{\delta/3}$.
Linearizing \eqref{ISH} around $h_{\ast}$ gives the stability rate $\lambda=2\delta$. Solutions remain $\varepsilon$-close to $\mathcal{M}$ for exponential long times.

We simulate  the reduced SDE \eqref{ISH} for $h(t)$ and visualizes metastability using the Euler-Maruyama method; see Fig. \ref{SH}. The amplitude $h(t)$ fluctuates near
$h_{\ast}=\sqrt{\delta/3}$ due to noise, which exhibits small fluctuations around $h_{\ast}$, with escapes occurring on timescales growing exponentially as $\varepsilon\rightarrow0$.

\begin{figure}
\centering
\includegraphics[width=10cm]{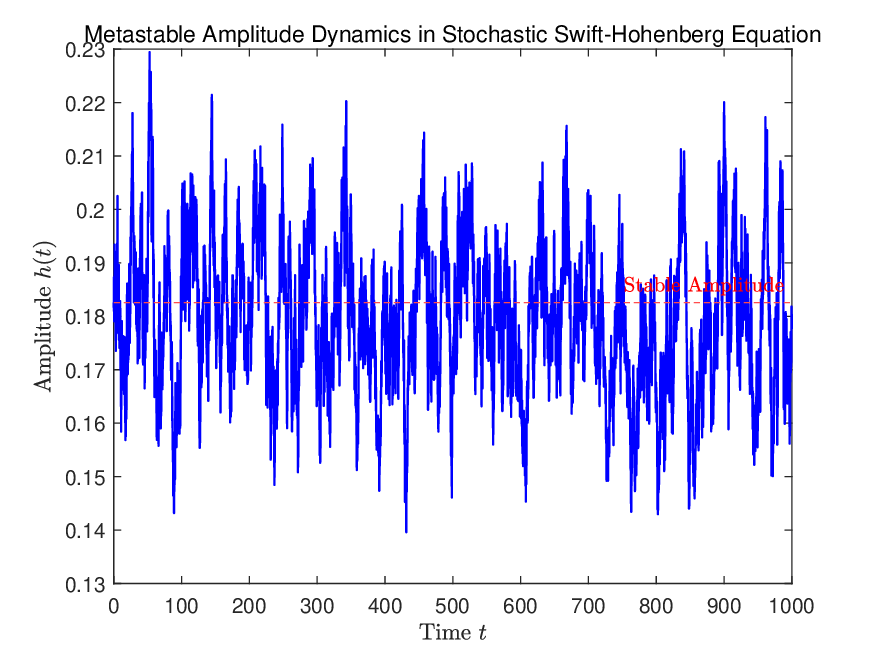}
\caption{Trajectory of $h(t)$ from \eqref{ISH}, the amplitude fluctuates near $h_{\ast}\approx0.183$ (dashed line), with rare escapes consistent with exponential metastability; Bifurcation parameter $\delta=0.1$; Noise intensity $\varepsilon=0.05$; Initial condition (stable equilibrium) $h_0=\sqrt{\delta/3}$; Total time $T=10^3$; Time step $dt=0.1$.}\label{SH}	
\end{figure}

\section{Conclusions and future challenges}\label{CF}
This work established an equivalence between the original stochastic partial differential equation and a coupled system comprising a finite-dimensional stochastic differential equation for the parameter $h$ on an approximate slow manifold and an infinite-dimensional evolution equation for the orthogonal component $v$. We provided a rigorous framework for reducing infinite-dimensional SPDEs to finite-dimensional SDEs on approximate slow manifolds, leveraging stochastic calculus and geometric considerations. Furthermore, we applied our results to the effective dynamics of stochastic models with multiplicative white noise, illustrated through four examples: the stochastic damped wave equation, the stochastic Allen-Cahn equation, the stochastic nonlinear Schr\"{o}dinger equation, and the stochastic Swift-Hohenberg equation.

While progress has been made in simplifying interface dynamics for multiplicative-noise-driven SPDEs, future work must address broader noise types, higher-dimensional complexity, and rigorous mathematical underpinnings to advance both theory and applications.

Singular limits, such as rigorously justifying dimension reduction procedures (e.g., vanishing interface width in sharp-interface approximations), demand precise analysis to bridge formal asymptotic expansions with stochastic calculus. Non-Markovian effects further complicate dynamics, requiring advanced tools to handle memory in systems driven by colored noise or fractional Brownian motion, where traditional It\^{o} calculus falls short. Meanwhile, extending results to high-dimensional systems, such as three-dimensional interfaces governed by curvature-driven motion or evolving under topological changes, introduces geometric and computational hurdles, particularly in preserving structure during discretization. Additionally, the influence of non-Gaussian noise (e.g., L\'evy flights or jump processes) on interface motion remains underexplored, necessitating new frameworks to characterize heavy-tailed fluctuations and discontinuous paths. These problems not only push the boundaries of stochastic analysis but also have profound implications for applications ranging from materials science to biological systems, where predictive modeling relies on reconciling rigorous theory with real-world complexity.

The study of stochastic interface dynamics in nonlinear SPDEs driven by multiplicative noise faces several future  challenges. First, extending theoretical frameworks to accommodate spatially or temporally correlated (colored), non-Gaussian (e.g., L\'evy), or degenerate noise is essential for modeling real-world systems where noise lacks Markovian or Gaussian structure, such as turbulent environments or fracture propagation in heterogeneous materials. Addressing curvature-driven dynamics, topological bifurcations (e.g., merging or pinching of interfaces), and complex geometries in higher dimensions demands advances in geometric stochastic calculus to handle evolving manifolds and singularities. Rigorous convergence proofs for reduced models require modern tools like regularity structures and paracontrolled distributions, alongside resolving renormalization challenges for singular interfaces with divergent energy densities. Strongly nonlinear regimes, where multiplicative noise couples non-trivially with high-order terms or interacts with conservation laws (e.g., mass/energy), necessitate novel analytical techniques to disentangle stochastic and deterministic feedback. The role of domain geometry, boundary conditions (e.g., rough or moving boundaries), and material heterogeneities in shaping interface dynamics remains underexplored, particularly in stochastic settings. Computationally, scalable algorithms for both the original SPDEs and their effective equations are urgently needed to bridge the gap between theoretical insights and real-world applications, such as predicting crack propagation in alloys or biofilm growth. A deeper understanding of how multiplicative noise selectively triggers or suppresses interfacial instabilities-such as fingering in multiphase flows or phase transitions in active matter will unify theory with experimental observations in materials science, developmental biology, climate change and  fluid dynamics. Ultimately, synthesizing mathematical rigor with empirical validation across disciplines will be pivotal in transforming abstract models into predictive tools for engineering and natural systems.

\section*{Acknowledgments}

 \noindent The authors are happy to thank Prof. Martin Hairer for the  insightful and inspiring discussions on interfaces.  This work was supported by Guangdong basic and applied basic research foundation 2025A1515012560.

\end{document}